\documentclass{amsart}
\usepackage{amsmath}
\usepackage{amssymb}
\usepackage{epsfig}
\usepackage{url}
\usepackage{hyperref}

\usepackage{enumerate}

\makeatletter
\@namedef{subjclassname@2010}{%
  \textup{2010} Mathematics Subject Classification}
\makeatother

\numberwithin{equation}{section}

\title{ERGODICITY OF $\Z^2$ EXTENSIONS OF IRRATIONAL ROTATIONS }

\author{Yuqing Zhang}

\address{ESI,
Boltzmanngasse 9,
A-1090 Vienna, Austria}

\email{\url{zhangy6@univie.ac.at}}

\newcommand{\R}{{\mathbb{R}}}

\newcommand{\Z}{{\mathbb{Z}}}

\newcommand{\N}{{\mathbb{N}}}

\newcommand{\E}{{\mathbb E}}

\newcommand{\T}{\mathbf{T}}

\newcommand{\x}{{\mathbf{X}}}
\providecommand{\abs}[1]{\left\lvert#1\right\rvert}
\providecommand{\norm}[1]{\left\lVert#1\right\rVert}

\newtheorem{theorem}{Theorem}[section]

\newtheorem{lem}[theorem]{Lemma}

\newtheorem{prop}[theorem]{Proposition}

\newtheorem{defi}[theorem]{Definition}

\newtheorem{remark}[theorem]{Remark}

\begin{document}

\begin{abstract}

Let $\T=[0,1)$ be the additive group of real numbers modulo $1$, $\alpha \in \T$ be an irrational number and
$t \in \T$.
We consider skew product extensions of irrational rotations by  $\Z^2$ determined by
$T \colon \T\times \Z^2 \to  \T\times \Z^2\quad$
$T(x,s_1,s_2)=\left(x+\alpha,\quad s_1+2\chi_{[0,\tfrac{1}{2})}(x)-1,\quad s_2+2\chi_{[0,\tfrac{1}{2})}(x+t)-1\right)$.
We study ergodic components of such extensions and use the results to  display irregularities in the uniform distribution of the sequence
 $\Z\alpha$.

\end{abstract}

\subjclass[2010]{Primary 37; Secondary 11}

\keywords{Cocycle, Ergodicity, Irrational rotation}

\maketitle

\section{Introduction}

The study of irrational rotations of the circle  leads to various questions in number theory and ergodic theory.
Let $\T=[0,1)$ be the additive group of real numbers modulo 1. Fix an  irrational $\alpha \in \T$  and let $t \in \T$
satisfy the condition that
neither $t$ nor $t+\tfrac{1}{2}$ is a multiple of $\alpha \mod 1$ .  Define a  map $f\colon \T\to
\Z$  by
\begin{equation}\label{fdefi}
f(x)=\left\{\begin{array}{c}
       1 \quad \text{for} \quad  0\leq x <\frac{1}{2};\\
       -1 \quad \text{for}  \quad \frac{1}{2}\leq x< 1
     \end{array}\right.
\end{equation}
and an irrational  rotation $T_0$ of $\T$ by
\begin{equation}
T_0x=x+\alpha \mod 1.
\end{equation}
Set $\x=\T\times \Z^2$ and define $T \colon \x \to  \x$ by
\begin{equation}\label{tdefi}
T(x,s_1,s_2)=\left(x+\alpha,\quad s_1+f(x),\quad s_2+f(x+t)\right).
\end{equation}
 $T$ is a skew product extension of irrational rotations on the circle by  $\Z^2$ determined by $f(x)$ and $t$.
  We  study ergodicity of $T$ on $\x$ relative to Haar measure, continuing a theme started by \cite{K1}, \cite{K2} of Schmidt and by \cite{V} of Veech.
  It is known that such property of skew product extensions of irrational rotations arises from irregularity of distribution of
$\Z\alpha$. As for the case of cylinder flows,
Oren in \cite{O} gave complete solution to the problem of  ergodicity of the map $F\colon \T\times E\to  \T\times E$ defined by
 $F(x,s)=(x+\alpha,s+\mathbf{1}_{[0,\beta)}(x)-\beta)$,
where $\beta \in \T$ and $E$ is the closed subgroup of $\R$ generated by 1 and $\beta$. Earlier, special cases were done by
Schmidt for $\beta=\tfrac{1}{2}$, $\alpha=\tfrac{\sqrt{5}-1}{4}$ in \cite{K2} and for $\beta=\tfrac{1}{2}$, $\alpha$ irrational in \cite{K1}.
Although ergodicity of cylinder flows has been understood thoroughly, due to the fact that $f(x)$ and $f(x+t)$ take on
 independent values,  the situation of $\Z^2$ extensions of irrational rotations appear to be more complicated.

Note that by definition \eqref{tdefi}, we have
\begin{equation}\label{tndefi}
T^n(x,s_1,s_2)=\left(x+n\alpha,\quad s_1+a_n(x),\quad s_2+a_n(x+t)\right), \quad \forall n \in \Z,
\end{equation}
where
\begin{equation}\label{an}
a_n(x)=\left\{\begin{split}
       &\sum_{i=0}^{n-1}f\left(x+i\alpha\right)=2\sum_{i=0}^{n-1}\chi_{[0,\frac{1}{2})}\left(x+i\alpha\right)-n, \quad \forall n \geq 1;\\
       &0, \quad  \text{for}   \quad  n=0;\\
       &-a_{-n}(T_0^{-n}x),\quad \forall n \leq -1.
     \end{split}\right.
\end{equation}
$t\in \Z\alpha$ and $t\in \Z\alpha + \tfrac{1}{2}$ are  excluded \emph{a priori}. To see this, note that for
nonnegative integer $m$, $|a_n(x+m\alpha)-a_n(x)|$ is bounded by  $2m$ because
\begin{align}\label{bound}
\left|a_n(x+m\alpha)-a_n(x)\right|&=\left|\sum_{i=0}^{m-1}f\left(x+n\alpha+i\alpha\right)-\sum_{i=0}^{m-1}f\left(x+i\alpha\right)\right|\\
&\leq \sum_{i=0}^{m-1}\left|f\left(x+n\alpha+i\alpha\right)\right|+\sum_{i=0}^{m-1}\left|f\left(x+i\alpha\right)\right|\leq 2m, \quad \forall n>m.
\end{align}
We also have from \eqref{fdefi} $f(x+\tfrac{1}{2})=-f(x)$ and therefore
\begin{equation}\label{symm}
a_n(x+\frac{1}{2})=-a_n(x), \quad \forall x \in \T, \quad \forall n.
\end{equation}
$|a_n(x+\tfrac{1}{2}+m\alpha)+a_n(x)|$ is bounded from above by  $2m$ thereof.

Also  note that $a_n(x+t) \equiv a_n(x) \mod 2$. The parity $a_n(x)$ is always the same as that of $n$ from \eqref{an}. Hence $T$
cannot be ergodic on the entire space $\x$.
We set $G=\{(s_1,s_2) \in \Z^2 \mid s_1\equiv s_2\mod 2\}$. $G$ is cocompact in $\Z^2$.

$a_n(x)$ satisfies the  additive cocycle equation
\begin{equation}\label{cocycle}
a_n\left(T_0^mx\right)-a_{n+m}(x)+a_m(x)=0, \quad \forall m,n \in \Z, \quad \forall x \in \T.
\end{equation}
Following  \cite[Definition 2.1]{K1} we have

\begin{defi}
 $\left(a,t\right)\colon \Z \times \T \to \Z^2$  defined by
\begin{equation}\label{cocycle2}
\left(a,t\right)(n,x)=(a_n(x),\quad a_n(x+t))
\end{equation}
is called a cocycle for $T_0$.
\end{defi}

\cite{K1} showed that ergodicity of $T$,  or equivalently,  ergodicity of the cocycle $(a,t)$ is determined  by the group $\E^2\left(a,t\right)$ of essential values of $(a,t)$.
Put $\overline{\Z^2}=\Z^2 \bigcup \{\infty\}$, the one point compactification of $\Z^2$.
We have the following definitions of essential values etc.

\begin{defi}
 Let $\mu$ be  Lebesgue  measure on $\T$. An element $(k_1,k_2) \in \overline{\Z^2}$ is called an essential value of $(a,t)$ if for every measurable set $ A \subset \T$ with $\mu(A)>0$, we have
\begin{equation}
\mu\left(\bigcup_{n \in \Z}\left( A\bigcap T_0^{-n}A\bigcap \left\{x\mid a_n(x)=k_1\right\}\bigcap\left\{x\mid a_n(x+t)=k_2\right\}\right)\right)>0,
\end{equation}
 We denote the set of essential values of $(a,t)$ by $\overline{\E^2}\left(a,t\right)$. 
\end{defi}

\begin{defi}
Set
$\E^2\left(a,t\right)=\overline{\E^2}\left(a,t\right)\bigcap \Z^2$. $(k_1,k_2) \in \overline{\E^2}\left(a,t\right)\setminus \E^2\left(a,t\right)$
only if $(k_1,k_2)$ does not lie in any compact subset of $\Z^2$.
\end{defi}

From \cite{K1} we derive the following properties
\begin{enumerate}
\item $\E^2\left(a,t\right)$ is  a closed  subgroup of $\Z^2$ under addition.
 $(k_1,k_2) \in \E^2\left(a,t\right)$ only if $k_1\equiv k_2 \mod 2$.

\item  $\left(a,t\right)$ is a coboundary (that is, $a_n(x)=c(T_0^{n}x)-c(x)$ for a measurable map $c \colon \T \to \Z)$ iff
$\overline{\E^2}\left(a,t\right)=\{(0,0)\}$.

 \end{enumerate}

We say that two cocycles $ \left(a,t\right), \left(b,t\right)\colon \Z \times \T \to \Z^2$  are cohomologous
 if $\left(a,t\right)-\left(b,t\right)$ is
a coboundary. In this case $\overline{\E^2}\left(a,t\right)=\overline{\E^2}\left(b,t\right)$.
Given a cocycle $ \left(a,t\right) \colon \Z \times \T \to \Z^2$, let
 $ \left(a,t\right)^{*} \colon \Z \times \T \to \Z^2/\E^2\left(a,t\right)$ be the corresponding quotient cocycle. We have
 the following important result from \cite[Lemma 3.10]{K1}:

 \begin{lem}
 $\E^2\left(a,t\right)^{*}=\{(0,0)\}$.
 \end{lem}

We say that the cocycle $\left(a,t\right)$ is regular if $\overline{\E^2}\left(a,t\right)^{*}=\{(0,0)\}$.
 $\left(a,t\right)$ is called nonregular if $\overline{\E^2}\left(a,t\right)^{*}=\{(0,0), \infty\}$.
If $\left(a,t\right)$ is regular, then
$\left(a,t\right)$ is cohomologous to a cocycle $\left(b,t\right)\colon \Z \times \T \to \E^2\left(a,t\right)$ and the latter is ergodic as a
cocycle with values in the closed subgroup $\E^2\left(a,t\right)$ (see \cite{K1}).
In particular, if $\E^2\left(a,t\right)$ is cocompact  in $\Z^2$ then $\left(a,t\right)$ is regular.

We utilize approach  devised in \cite{K1}, \cite{O} to prove the following theorems:

\begin{theorem}\label{theorem1}
For arbitrary irrational $\alpha \in \T$,  the group of essential values
  $\E^2\left(a,t\right)$ of the cocycle $\left(a,t\right)$  defined in \eqref{cocycle2} is $G=\{(s_1,s_2) \in \Z^2 \mid s_1\equiv s_2\mod 2\}$
   for almost all $t \in \T$. In particular,
 $\left(a,t\right)$ is  regular for almost all $t \in \T$.
 \end{theorem}

\begin{theorem}\label{theorem2}
If $\alpha$ is badly approximable, then  the group of essential values
  $\E^2\left(a,t\right)$ is $G$ if and only if  $t\notin \Z\alpha$ and $t\notin \Z\alpha + \tfrac{1}{2}$.
\end{theorem}

\section{ Period approximating sequences, Partial convergents and other preliminaries}

 For $x \in \R$ we denote  the closest integer to $x$ by $[x]$, denote $x-[x]$ by $\langle x\rangle$
and denote $|x-[x]|$ by $\lVert x \rVert$. We assume $n$ to be nonnegative.

According to \eqref{an} $a_n(x)$ is locally constant except for points of discontinuities of $+2$ at $0,-\alpha, -2\alpha, \dotsc, -(n-1)\alpha$ and
points of discontinuities of $-2$ at
$\tfrac{1}{2}, \tfrac{1}{2}-\alpha, \tfrac{1}{2}-2\alpha, \dotsc, \tfrac{1}{2}-(n-1)\alpha$.
$a_n(x+t)$ is locally constant except for points of discontinuities of $+2$ at $-t, -t-\alpha, -t-2\alpha, \dotsc, -t-(n-1)\alpha$ and
points of discontinuities of $-2$ at
$\tfrac{1}{2}-t, \tfrac{1}{2}-t-\alpha,  \dotsc, \tfrac{1}{2}-t-(n-1)\alpha$.

If we set
\begin{equation}
S_n(x)=\sum_{i=0}^{n-1}\chi_{[0,\frac{1}{2})}\left(x+i\alpha\right)=\#
\left\{i\mid 0\leq i \leq n-1; \quad x+i\alpha \in [0,\frac{1}{2})\right\},
\end{equation}
then from \eqref{an}
\begin{equation}
a_n(x)=2S_n(x)-n.
\end{equation}
The concept of essential values corresponds to that of periods in \cite{O}. We have the following definition:

\begin{defi}
A period approximating sequence is a sequence $\left\{(n_l,A_l)\right\}_{l=1}^\infty$ where
\begin{enumerate}
\item $A_l \subset \T$, each $A_l$ is measurable;
\item $a_{n_l}$ is constant on  both $A_l$ and $A_l+t$, that is, $a_{n_l}(A_l)=k_1, a_{n_l}(A_l+t)=k_2 \quad \forall n_l$;
\item $\inf_l \mu(A_l)>0$;
\item $\norm{n_l\alpha} \to 0$.
\end{enumerate}
\end{defi}

The next lemma shows that a period approximating sequence defines  an element in $\E^2\left(a,t\right)$.

\begin{lem}
If there exists a period approximating sequence $\left\{(n_l,A_l)\right\}_{l=1}^\infty$ such that $a_{n_l}(A_l)=k_1$, $a_{n_l}(A_l+t)=k_2$, $\forall n_l$,
then $(k_1,k_2) \in \E^2\left(a, t\right)$.
\end{lem}

\begin{proof}
Set \[B=\limsup_{l\to\infty} A_l=\bigcap_{l=1}^\infty\bigcup_{i=l}^\infty A_i. \] $\mu(B)>0$ because $\inf_l \mu(A_l)>0$ and $\mu(\T)=1$.

For arbitrary $A \subset \T$ with $\mu(A)>0$,
there exists $m \in \Z$ and $A' \subset A$
such that $\mu(A')>0$ and  $T_0^m A'\subset B$ because the action $T_0$ is ergodic. Hence
\begin{equation}
\mu\left(B\bigcap T_0^m A'\right)=\mu\left(\bigcap_{l=1}^\infty\bigcup_{i=l}^\infty( A_i\bigcap T_0^m A')\right)=\mu\left(T_0^m A'\right)>0,
\end{equation}
 hence  there exists a subsequence $\{n_l'\}$ of $\{n_l\}$ such that for each $n_l'$,
 there exists a measurable set $A'_{n_l'}\subset A'$ with $\mu(A'_{n_l'})>0$ and
\begin{equation}
a_{n_l'}(T_0^mx)=k_1,\quad a_{n_l'}(T_0^mx+t)=k_2, \quad \forall x \in A'_{n_l'}.
\end{equation}
Note that
\begin{align}
\abs{a_{n_l'}(T_0^mx)-a_{n_l'}(x)}&=\abs{\sum_{i=0}^{n_l'-1}f\left(x+i\alpha+m\alpha\right)-\sum_{i=0}^{n_l'-1}f\left(x+i\alpha\right)}\\
&=\left|\sum_{i=0}^{m-1}f\left(x+i\alpha+n_l'\alpha\right)-\sum_{i=0}^{m-1}f\left(x+i\alpha\right)\right|, \nonumber
\end{align}
\begin{align}
\left|a_{n_l'}(T_0^mx+t)-a_{n_l'}(x+t)\right|=\left|\sum_{i=0}^{m-1}f\left(x+i\alpha+n_l'\alpha+t\right)-\sum_{i=0}^{m-1}f\left(x+i\alpha+t\right)\right|,
\end{align}
\begin{equation}
\lim\norm{n_l'\alpha}= 0,
\end{equation}
as well as the fact that $m$ is fixed and depends on $A$ only, we deduce that
there exists some $n_l'$ and $A'' \subset A' \subset A$ with   $\mu(A'')>0$ such that
\begin{equation}
a_{n_l'}(T_0^mx)=a_{n_l'}(x)=k_1, \quad a_{n_l'}(T_0^mx+t)=a_{n_l'}(x+t)=k_2, \quad \forall x \in A''.
\end{equation}
Hence we have
\begin{equation}
\mu\left( A\bigcap T_0^{-n_l'}A\bigcap \left\{x\mid a_{n_l'}(x)=k_1\right\}\bigcap\left\{x\mid a_{n_l'}(x+t)=k_2\right\}\right)>0.
\end{equation}
$(k_1,k_2) \in \E^2\left(a, t\right)$.
\end{proof}

We record the statement of the Denjoy-Koksma inequality \cite[Lemma 2]{O} here, which plays a fundamental role in the proof.

\begin{lem}[Denjoy-Koksma]\label{denjoy}
If $p \in \N,q \in \N$ satisfy
\[\abs{\alpha-\frac{p}{q}}<\dfrac{1}{q^{2}} \quad \text{ and } \quad  (p,q)=1,\]
 then $\abs{a_q(x)}<4$, $\forall x \in \T$, where $a_q(x)$ is defined in \eqref{an}.
\end{lem}

It follows from the proof of the above lemma that every interval of the form $\left[\tfrac{i}{q},\tfrac{i+1}{q}\right)$ contains exactly one of the
points $j\alpha$ for $0\leq i, j \leq q-1$. In other words, the points $j\alpha$ ($0\leq  j \leq q-1$) are uniformly distributed on the unit circle.

We rely on numerous facts concerning continued fractions stated in texts such as \cite{K}. A considerable portion of our approach is borrowed from
\cite{O}. However, here we need to construct period approximating sequence  $\left\{(n_l,A_l)\right\}_{l=1}^\infty$ such that
$a_{n_l}\left(A_l\right)$ and $a_{n_l}\left(A_l+t\right)$ take on \emph{independent} values whereas predecessors of this paper only deal with cylinder flows.

We denote by $\left[a_0; a_1,a_2, \ldots,\right]$ the continued fraction of $\alpha$ and call
 the $a_i$ the partial quotients of $\alpha$. Denote by $\tfrac{p_k}{q_k}$ the $k$th partial convergent of $\alpha$
 where $k\geq 0$. It is known from
\cite{K} that
\begin{equation}
\dfrac{p_k}{q_k}=[a_0; a_1,a_2, \ldots,a_k];
\end{equation}
\begin{equation}\label{continued}
\norm{q_k\alpha}<\dfrac{1}{q_{k+1}}<\dfrac{1}{q_k};
\end{equation}
\begin{equation}\label{continued22}
\min_{q_k\leq q<q_{k+1}} \norm{q\alpha}=\norm{q_k\alpha}>\dfrac{1}{q_k+q_{k+1}}>\dfrac{1}{2q_{k+1}}.
\end{equation}
\begin{equation}\label{continued55}
q_kp_{k-1}-p_kq_{k-1}=(-1)^{k}.
\end{equation}
Set
\begin{equation}
D(\alpha)=\left\{q_k\mid \dfrac{p_k}{q_k} \text{ is a partial convergent of } \alpha \right\};
\end{equation}
\begin{equation}\label{next}
q^+=\min\{ q' \in D(\alpha)\mid q'>q\}, \quad \forall q \in D(\alpha).
\end{equation}
Adopting arguments on   \cite[Page 229-230]{K1} we are able to prove the following lemma which constitutes the
first step in the entire proof:

\begin{lem}\label{first}
\begin{equation}
\E^2\left(a,t\right)\bigcap\left\{(1,3),(1,-3),(1,1), (1,-1),(3,1),(3,-1),(3,3), (3,-3)\right\}\neq \emptyset.
\end{equation}
\end{lem}

\begin{proof}

From \eqref{continued55} we derive that there are infinitely many odd $q \in D(\alpha)$.
For such $q \in D(\alpha)$, the Denjoy-Koksma inequality applies. In addition, from \eqref{an} we see that $a_q(x)$
can only be odd, that is,  $a_q(x)$ can only be $\pm3$ or $\pm1$.

Consequently there exists a period approximating sequence  $\left\{(q_l,A_l)\right\}_{l=1}^\infty$
such that $q_l \in D(\alpha)$,
\begin{enumerate}
\item $A_l \subset \T$;
\item $a_{q_l}$ is constant on  both $A_l$ and $A_l+t$, $a_{q_l}(A_l)=k_1, a_{q_l}(A_l+t)=k_2 \quad \forall n_l$;
\item $\inf_l \mu(A_l)>0$;
\item $\|q_l\alpha\|\rightarrow 0$
\end{enumerate}
and $(k_1,k_2) \in \left\{(1,3),(1,-3),(1,1), (1,-1),(3,1),(3,-1),(3,3), (3,-3)\right\}$\\
$ \bigcup
\left\{-(1,3),-(1,-3),-(1,1),- (1,-1),-(3,1),-(3,-1),-(3,3), -(3,-3)\right\}$.
The proof is complete by noting that $\E^2\left(a,t\right)$ is a group under addition.
\end{proof}

A major difficulty to prove Theorem \ref{theorem1} is therefore to show that
$\E^2\left(a,t\right)$ is not isomorphic to $\Z$.
We aim to show that $\E^2\left(a,t\right)$ is  $G$  for almost all $t$.
This is done by using period approximating sequences.
We derive from properties of continued fractions the following lemma:

\begin{lem}\label{continued33}
For any nonzero $q \in D(\alpha)$, we have
\begin{equation}
\min \left\{\norm{\tfrac{1}{2}-j\alpha} \mid |j|<q\right\}\geq \dfrac{1}{24q}.
\end{equation}
\end{lem}

\begin{proof}
We always have
\begin{equation}
 \left\|\frac{1}{2}-j\alpha\right\|\geq \frac{\|2(\frac{1}{2}-j\alpha)\|}{2}= \frac{\|2j\alpha\|}{2}.
\end{equation}
We consider five cases separately under the assumption that
 $0<|j|<q$.

\textbf{Case 1:} $q^+\geq 3q$, then since $||2j|-q|<q$ from $0<|j|<q$, we have $\|(|2j|-q)\alpha\|>\frac{1}{2q}$ from \eqref{continued22} and
\begin{equation}
\|2j\alpha\|=\|(|2j|-q)\alpha+q\alpha\|\geq\|(|2j|-q)\alpha\|-\|q\alpha\|>\frac{1}{2q}-\frac{1}{q^{+}}>\frac{1}{2q}-\frac{1}{3q}=\frac{1}{6q}.
\end{equation}
Here we also used the inequality $\|q\alpha\|<\dfrac{1}{q^{+}}$ from \eqref{continued}.

\textbf{Case 2:} If $q^+< 3q$ and $q^{++}< 3q$, then since $|2j|<2q\leq q^{++}$, we have from \eqref{continued22}
\begin{equation}
\|2j\alpha\|\geq\|q^+\alpha\|\geq\frac{1}{2q^{++}}>\frac{1}{6q}.
\end{equation}

\textbf{Case 3:} If $q^+< 3q$, $q^{++}\geq 3q$ and $|q^+-|2j||< q$, then we have $\|(|2j|-q^+)\alpha\|>\frac{1}{2q}$ from \eqref{continued22} and
\begin{equation}
\|2j\alpha\|=\|(|2j|-q^+)\alpha+q^+\alpha\|\geq\|(|2j|-q^+)\alpha\|-\|q^+\alpha\|>\frac{1}{2q}-\frac{1}{q^{++}}>\frac{1}{2q}-\frac{1}{3q}=\frac{1}{6q}.
\end{equation}

\textbf{Case 4:} If $q^+< 3q$, $q^{++}\geq 3q$,  $|q^+-|2j||\geq q$ and  $|2j|\leq q$, then from \eqref{continued22} we get
\begin{equation}
\|2j\alpha\|\geq\|q\alpha\|>\frac{1}{2q^{+}}\geq \frac{1}{6q}.
\end{equation}

\textbf{Case 5:} If $q^+< 3q$, $q^{++}\geq 3q$,   $|q^+-|2j||\geq q$ and $|2j|> q$, then
\begin{equation}
q^+-|4j|<3q-2q=q, \quad 2q-q^+>2q-3q=-q;
\end{equation}
\begin{equation}
|2j|\leq q^+-q \rightarrow q^+-|4j|\geq q^+-2(q^+-q)=2q-q^+>-q;
\end{equation}
hence $|q^+-|4j||<q$ and from \eqref{continued22}
\begin{equation}
\|4j\alpha\|=\|(q^+-|4j|)\alpha-q^+\alpha\|\geq\|(q^+-|4j|)\alpha\|-\|q^+\alpha\|>\frac{1}{2q}-\frac{1}{q^{++}}\geq \frac{1}{6q};
\end{equation}
and $\|2j\alpha\|\geq\dfrac{\|4j\alpha\|}{2}$. The inequality is established.
\end{proof}

\section{Proof of main theorems}

Following \cite{O} we set for each $q \in D(\alpha)$
\begin{align}\label{smallest}
\epsilon(q) = & q\cdot\min\left\{ \left\|-t-j\alpha\right\|\mid |j|<q \right\};\\
\theta(q) = & q\cdot\min\left\{ \left\|\tfrac{1}{2}-t-j\alpha\right\|\mid |j|<q \right\}. \nonumber
\end{align}
We immediately derive that  $\epsilon(q)<1$ and $\theta(q)<1$ from the  proof of the Denjoy-Koksma inequality.

\begin{prop}\label{bouded}
If \begin{equation}\label{continued44}
\limsup_{\substack{q \in  D(\alpha)   \\ q \rightarrow \infty}} \min\left\{\epsilon(q),\theta(q)\right\}>0,
\end{equation}
then
$\E^2\left(a,t\right)=\left\{(k_1,k_2) \in \Z^2\mid k_1\equiv k_2 \mod 2 \right\}=G$.
\end{prop}

\begin{proof}

Let $\{q_n\}_{n=1}^\infty \subset D(\alpha)$ be such that $\min\left\{\epsilon(q_n),\theta(q_n)\right\}>\delta>0$, $\forall n$.

Recall $a_{q_n}(x)$ as set in \eqref{an} is locally constant except for points of discontinuities of $+2$ at $0,-\alpha, -2\alpha, \dotsc, -(q_n-1)\alpha$ and
points of discontinuities of $-2$ at
$\tfrac{1}{2}, \tfrac{1}{2}-\alpha, \tfrac{1}{2}-2\alpha, \dotsc, \tfrac{1}{2}-(q_n-1)\alpha$.
 $a_{q_n}(x+t)$ is locally constant except for points of discontinuities of $+2$ at $-t, -t-\alpha, -t-2\alpha, \dotsc, -t-(q_n-1)\alpha$ and
points of discontinuities of $-2$ at
$\tfrac{1}{2}-t, \tfrac{1}{2}-t-\alpha,  \dotsc, \tfrac{1}{2}-t-(q_n-1)\alpha$.

For fixed $n$, let $I_1,I_2,\ldots,I_{4q_n}$ denote the intervals of constancy of both $a_{q_n}(x)$
and $a_{q_n}(x+t)$ in cyclic order. Since $a_{q_n}(\cdot)$ takes on
at most four values by Lemma \ref{denjoy}, there exists a union of intervals, $A_n$, such that $a_{q_n}(x)$ and
$a_{q_n}(x+t)$ are constant on $A_n$ and $\mu(A_n)\geq\tfrac{1}{16}$.
Let  $A_n'$ be the union of  intervals proximal on the right to those of  $A_n$.
Note that the distance between any discontinuities of $a_{q_n}(x)$ and $a_{q_n}(x+t)$ is given by
$\|(i-j)\alpha\|$ or $\|\tfrac{1}{2}+(i-j)\alpha\|$ or $\|-t+(i-j)\alpha\|$ or $\|\tfrac{1}{2}-t+(i-j)\alpha\|$
for $0\leq i,j\leq q_n-1$. From \eqref{continued22}, Lemma \ref{continued33} and \eqref{continued44},
we have that $\min\left\{\tfrac{1}{24q_n},\tfrac{\epsilon(q_n)}{q_n}, \tfrac{\theta(q_n)}{q_n}\right\}$ is a lower bound
for the lengths $|I_i|$, $i=1,2,\dotsc, 4q_n$. Since every interval of length $\tfrac{2}{q_n}$ must contain a +2
discontinuity by discussion following Lemma \ref{denjoy},  we have   $|I_i|<\tfrac{2}{q_n}$. Therefore we have
\begin{equation}
\frac{|I_i|}{|I_j|}>\frac{1}{2}\min\left\{\frac{1}{24},\epsilon(q_n), \theta(q_n)\right\}, \quad 1\leq i, j \leq 4q_n.
\end{equation}
By setting $\epsilon=\min\left\{\frac{1}{24},\delta\right\}$, we thus have
$\mu(A_n')\geq\tfrac{1}{2}\epsilon\mu(A_n)\geq\tfrac{1}{32}\epsilon, \; \forall n$.
$\left(a,t\right)(q_n,x)=\left(a_{q_n}(x),a_{q_n}(x+t)\right)$ can take on $A_n'$ only the values $\left(a_{q_n}(A_n)\pm2,a_{q_n}(A_n+t)\right)$
or $\left(a_{q_n}(A_n),a_{q_n}(A_n+t)\pm2 \right)$ since each interval of $A_n'$ is proximal on the right to one of  $A_n$.
We can find $A_n''\subset A_n'$ such that
$a_{q_n}(x)$ and $a_{q_n}(x+t)$ are both constant on $A_n''$,  $\mu(A_n'')\geq\frac{1}{128}\epsilon$ and
\begin{equation}
\left(a_{q_n}(A_n''),a_{q_n}(A_n''+t) \right)=\left(a_{q_n}(A_n)\pm2,a_{q_n}(A_n+t)\right)\text{  or }
\left(a_{q_n}(A_n),a_{q_n}(A_n+t)\pm2 \right).
\end{equation}
We assume that  $a_{q_n}(A_n)=1$ and $a_{q_n}(A_n+t)=3$, that is $(1,3)$  lies in $\E^2\left(a,t\right)$. We
prove  both $(2,0)$ and $(0,2)$ lie in $\E^2\left(a,t\right)$.
Other possibilities can be treated
analogously.

\textbf{Case 1:}

Suppose we have  $(3,3)$ and  $(1,3)$ both lie in $\E^2\left(a,t\right)$ as a result of the above arguments.
 $(\pm2,0)$ lies in $\E^2\left(a,t\right)$ because  $\E^2\left(a,t\right)$ is a subgroup of $\Z^2$.

Moreover,  there exists a period approximating sequence  $\left\{(q_n,A_n)\right\}_{n=1}^\infty$ which defines
$(1,3) \in \E^2\left(a,t\right)$. Namely we have
\begin{enumerate}
\item $A_n \subset \T$;
\item $a_{q_n}$ is constant on  both $A_n$ and $A_n+t$, $a_{q_n}(A_n)=1, a_{q_n}(A_n+t)=3, \quad \forall n$;
\item $\inf_n \mu(A_n)>0$;
\item $\|q_n\alpha\|\rightarrow 0$.
\end{enumerate}
Therefore there exists a period approximating sequence $\left\{(q_n',B_n')\right\}_{n=1}^\infty$
which defines
$(k,1) \in \E^2\left(a,t\right)$ for some $k \in \{\pm1, \pm3\}$. Namely we have
\begin{enumerate}
\item  $\{q_n'\}$ is a subsequence of $\{q_n\}$, $B_n'+t  \subset A_n'$,  $\mu(B_n')\geq \tfrac{1}{4}\mu(A_n')$;
\item $a_{q_n'}$ is constant on  both $B_n'$ and $B_n'+t$, $a_{q_n'}(B_n')=k$, \\ $a_{q_n'}(B_n'+t)=a_{q_n'}(A_n')=1,\quad \forall n'$;
\item $\inf_{n'} \mu(B_n')>0$;
\item $\|q_n'\alpha\|\rightarrow 0$.
\end{enumerate}
\begin{equation}
(3,3) \in \E^2\left(a,t\right) \text{ and } (2,0) \in \E^2\left(a,t\right)\rightarrow (k,3) \in \E^2\left(a,t\right);
\end{equation}
\begin{equation}
(k,1) \in \E^2\left(a,t\right) \text{ and } (k,3) \in \E^2\left(a,t\right)\rightarrow (0,2) \in \E^2\left(a,t\right).
\end{equation}
Consequently both $(2,0)$ and  $(0,2)$ lie in $\E^2\left(a,t\right)$.

\textbf{Case 2:}

Suppose we have  $(-1,3)$ and  $(1,3)$ both lie in $\E^2\left(a,t\right)$.
$(\pm2,0)$ lies in $\E^2\left(a,t\right)$ because  $\E^2\left(a,t\right)$ is a subgroup of $\Z^2$.

Moreover,  there exists a period approximating sequence  $\left\{(q_n,A_n)\right\}_{n=1}^\infty$ which defines
$(1,3) \in \E^2\left(a,t\right)$. Namely we have
\begin{enumerate}
\item $A_n \subset \T$;
\item $a_{q_n}$ is constant on  both $A_n$ and $A_n+t$, $a_{q_n}(A_n)=1, a_{q_n}(A_n+t)=3, \quad \forall n$;
\item $\inf_n \mu(A_n)>0$;
\item $\|q_n\alpha\|\rightarrow 0$.
\end{enumerate}
Therefore there exists a period approximating sequence $\left\{(q_n',B_n')\right\}_{n=1}^\infty$ which defines
$(k,1) \in \E^2\left(a,t\right)$ for some $k \in \{\pm1, \pm3\}$. Namely we have
\begin{enumerate}
\item $\{q_n'\}$ is a subsequence of $\{q_n\}$, $B_n'+t  \subset A_n'$, $\mu(B_n')\geq \tfrac{1}{4}\mu(A_n')$;
\item $a_{q_n'}$ is constant on  both $B_n'$ and $B_n'+t$, $a_{q_n'}(B_n')=k$, \\ $a_{q_n'}(B_n'+t)=a_{q_n'}(A_n')=1,\quad \forall n'$;
\item $\inf_{n'} \mu(B_n')>0$;
\item $\|q_n'\alpha\|\rightarrow 0$.
\end{enumerate}
\begin{equation}
(1,3) \in \E^2\left(a,t\right) \text{ and } (2,0) \in \E^2\left(a,t\right)\rightarrow (k,3) \in \E^2\left(a,t\right);
\end{equation}
\begin{equation}
(k,1) \in \E^2\left(a,t\right) \text{ and } (k,3) \in \E^2\left(a,t\right)\rightarrow (0,2) \in \E^2\left(a,t\right).
\end{equation}
Consequently both $(2,0)$ and  $(0,2)$ lie in $\E^2\left(a,t\right)$.

\textbf{Case 3:}

Suppose we have  $(1,1)$ and  $(1,3)$ both lie in $\E^2\left(a,t\right)$. $(0,2)$ lies in $\E^2\left(a,t\right)$.
$(2,2)$ also lies in $\E^2\left(a,t\right)$ and therefore $(2,0)$ lies in $\E^2\left(a,t\right)$.

In all cases we have shown both $(2,0)$ and  $(0,2)$ lie in $\E^2\left(a,t\right)$.
Along with the assumption that $(1,3)$  lies in $\E^2\left(a,t\right)$,
we derive that
$\E^2\left(a,t\right)=G$ as desired.
\end{proof}

\begin{remark}
For arbitrary $\alpha$ the set of $t$ satisfying \eqref{continued44} has full Lebesgue measure.
Therefore for almost all $t \in \T$, we have $\E^2\left(a,t\right)=G$ and
Theorem \ref{theorem1} is established.
\end{remark}

Next we prove Theorem \ref{theorem2}. Note that $\alpha$ is badly approximable if and only if its partial quotients
are bounded.

\begin{prop}\label{bouded2}
If $\alpha$ is badly approximable and
\begin{equation}\label{continued6}
\lim_{\substack{q \in  D(\alpha)   \\ q \rightarrow \infty}} \min\left\{\epsilon(q),\theta(q)\right\}=0,
\end{equation}
then
$t \in \Z\alpha$ or $t \in \Z\alpha+\tfrac{1}{2}$.
\end{prop}

\begin{proof}
For each $q \in D(\alpha)$, let $|i_{q}|<q, |j_{q}|<q$ be such that
\begin{equation*}
\epsilon(q) =  q\left\|-t-i_{q}\alpha\right\|, \quad
\theta(q) =  q\left\|\tfrac{1}{2}-t-j_{q}\alpha\right\|.
\end{equation*}
 Then we have from the assumption of the proposition
 \begin{equation*}
\lim_{\substack{q \in  D(\alpha)   \\ q \rightarrow \infty}} \min\left\{q\left\|-t-i_{q}\alpha\right\|,\quad q\left\|\tfrac{1}{2}-t-j_{q}\alpha\right\|\right\}=0.
\end{equation*}
Because $\alpha$ is badly approximable, $\tfrac{q^+}{q}$ and $\tfrac{q^{++}}{q}$  have a uniform upper bound and
\begin{equation*}
\lim_{\substack{q \in  D(\alpha)   \\ q \rightarrow \infty}} \min\left\{q^{++}\left\|-t-i_{q}\alpha\right\|,\quad q^{++}\left\|\tfrac{1}{2}-t-j_{q}\alpha\right\|\right\}=0.
\end{equation*}
Also we have for arbitrary $n_1$ and $n_2$ the following inequalities:
\begin{equation}\label{eq1}
  \left\|n_1\alpha-n_2\alpha\right\|\leq \left\|-t-n_1\alpha\right\|+ \left\|-t-n_2\alpha\right\|, 
\end{equation}
\begin{equation}\label{eq2}
   \left\|\tfrac{1}{2}+n_1\alpha-n_2\alpha\right\|\leq \left\|\tfrac{1}{2}-t-n_1\alpha\right\|+ \left\|-t-n_2\alpha\right\|.
\end{equation}
If we have $q^{++}\left\|-t-i_{q^+}\alpha\right\|<\tfrac{1}{100}$ and $q^{++}\left\|\tfrac{1}{2}-t-j_{q}\alpha\right\|<\tfrac{1}{100}$, then by \eqref{eq2}
\[q^{++}\left\|\tfrac{1}{2}+i_{q^+}\alpha-j_{q}\alpha\right\|<\dfrac{1}{50}.\]
Because
\begin{equation*}
|i_{q^+}-j_{q}|\leq |i_{q^+}|+|j_{q}|<q^++q\leq q^{++},
\end{equation*}
this contradicts  Lemma \ref{continued33}, which asserts that $q^{++}\left\|\tfrac{1}{2}+i_{q^+}\alpha-j_{q}\alpha\right\|\geq\dfrac{1}{24}$.
Hence  we have
\begin{equation*}
\lim_{\substack{q \in  D(\alpha)   \\ q \rightarrow \infty}}  q\left\|-t-i_{q}\alpha\right\|=0
\quad \text{ or } \quad \lim_{\substack{q \in  D(\alpha)   \\ q \rightarrow \infty}}q\left\|\tfrac{1}{2}-t-j_{q}\alpha\right\|=0.
\end{equation*}
Suppose we have $\lim_{\substack{q \in  D(\alpha)   \\ q \rightarrow \infty}}  q\left\|-t-i_{q}\alpha\right\|=0$, then by \eqref{eq1}
\begin{equation*}
\lim_{\substack{q \in  D(\alpha)   \\ q \rightarrow \infty}} q^{++}\left\|i_{q^+}\alpha-i_{q}\alpha\right\|=0.
\end{equation*}
From \eqref{continued22} we derive that for $q$ large enough $i_{q^+}=i_{q}$, that is, $i_{q}$ is constant.
Hence $t \in \Z\alpha$.

Suppose we have $\lim_{\substack{q \in  D(\alpha)   \\ q \rightarrow \infty}}  q\left\|\tfrac{1}{2}-t-j_{q}\alpha\right\|=0$, then
\begin{equation*}
\lim_{\substack{q \in  D(\alpha)   \\ q \rightarrow \infty}} q^{++}\left\|j_{q^+}\alpha-j_{q}\alpha\right\|=0.
\end{equation*}
From \eqref{continued22} we derive that for $q$ large enough $j_{q^+}=j_{q}$, that is, $j_{q}$ is constant.
Hence $t \in \Z\alpha+\tfrac{1}{2}$.
\end{proof}

\begin{remark}
When $\alpha$ is not badly approximable, Merrill \cite{M} showed that if $t$ belongs to an uncountable set of zero measure containing numbers well approximable by multiples of $\alpha$, the  cocycle $v=\chi_{[0,t)}-\chi_{[\tfrac{1}{2},\tfrac{1}{2}+t)}$ is a coboundary.
 This implies $\E^2\left(a,t\right)=\left\{(k,k) \mid k\in \Z \right\}$. Similarly,
 If $t+\tfrac{1}{2}$ belongs to an uncountable set of zero measure containing numbers well approximable by multiples of $\alpha$, then
$\E^2\left(a,t\right)=\left\{(k,-k) \mid k\in \Z \right\}$.
\end{remark}

\subsection*{Acknowledgements}

The author is supported by
Austrian Science Fund (FWF) Grant NFN S9613.
The author thanks Professor Klaus Schmidt for  helpful discussions and the ESI for hospitality and partial support.

\end{document}